\documentclass[1p,final]{elsarticle}
\usepackage{amssymb,color,pslatex,natbib}
\usepackage{amsthm}
 \usepackage{stmaryrd}
\usepackage{amssymb}
 \usepackage{amsfonts}
  \usepackage{mathrsfs}
  \usepackage{latexsym,bm}
   \usepackage{yfonts}
 \usepackage[all,poly,knot]{xy}
 \usepackage{tipa}
 \usepackage{fancyhdr}

\usepackage[latin1]{inputenc}
\usepackage{graphicx}

 \newtheorem{theorem}{Theorem}[section]
\newtheorem{Legendre theorem}{Legendre Theorem}[section]
\newtheorem{lemma}[theorem]{Lemma}
\newtheorem{conjecture}[theorem]{Conjecture}
\newtheorem{corollary}[theorem]{Corollary}
\newtheorem{example}[theorem]{Example}
\newtheorem{remark}[theorem]{Remark}
\newtheorem{proposition}[theorem]{Proposition}
\newtheorem{definition}[theorem]{Definition}

\newcommand{\Tr}{{\rm Tr}}
\newcommand{\gf}{ {{\mathbb F}} }

\begin{document}

\begin{frontmatter}

%% Title, authors and addresses

%% use the tnoteref command within \title for footnotes;
%% use the tnotetext command for the associated footnote;
%% use the fnref command within \author or \address for footnotes;
%% use the fntext command for the associated footnote;
%% use the corref command within \author for corresponding author footnotes;
%% use the cortext command for the associated footnote;
%% use the ead command for the email address,
%% and the form \ead[url] for the home page:
%%
%% \title{Title\tnoteref{label1}}
%% \tnotetext[label1]{}
%% \author{Name\corref{cor1}\fnref{label2}}
%% \ead{email address}
%% \ead[url]{home page}
%% \fntext[label2]{}
%% \cortext[cor1]{}
%% \address{Address\fnref{label3}}
%% \fntext[label3]{}

\title{Permutation trinomials over $\gf_{2^m}$: a corrected version}
\tnotetext[fn1]{*Corresponding author.  P. Yuan's research was supported by the NSF of China (Grant No. 11271142, 11671153).}

%% use optional labels to link authors explicitly to addresses:
%% \author[label1,label2]{<author name>}
%% \address[label1]{<address>}
%% \address[label2]{<address>}
\author[ypz]{Danyao Wu}
\ead{wudanyao@163.com}
\author[ypz]{Pingzhi Yuan*}
\ead{yuanpz@scnu.edu.cn}
\author[cding]{Cunsheng Ding}
\ead{cding@ust.hk}
\author[ypz]{Yuzhen Ma}
\ead{617875902@qq.com}

%\cortext[zcz]{Corresponding author}
\address[ypz]{School of Mathematics, South China Normal University, Guangzhou 510631, China}
\address[cding]{Department of Computer Science
                                                  and Engineering, The Hong Kong University of Science and Technology,
                                                  Clear Water Bay, Kowloon, Hong Kong, China}

\begin{abstract}
Permutation polynomials are an interesting subject of mathematics and have applications in other areas of
mathematics and engineering.
In this paper,   we determine all permutation
trinomials over $\gf_{2^m}$ in Zieve's paper \cite{Zi13}. We prove a conjecture proposed by Gupta and Sharma in \cite{GS16} and obtain some new  permutation trinomials over $\gf_{2^m}$. Finally,  we show that some classes of permutation trinomials with parameters are QM equivalent to some known permutation trinomials.

\end{abstract}

\begin{keyword}
polynomial \sep permutation polynomial.
%% PACS codes here, in the form: \PACS code \sep code

%% MSC codes here, in the form: \MSC code \sep code
%% or \MSC[2008] code \sep code (2000 is the default)
\MSC  11C08 \sep 12E10

\end{keyword}

\end{frontmatter}

\section{Introduction}

Let  $\gf_q$ be the finite field with $q$ elements,  where $q$ is a prime power, and
let $\gf_q[x]$
be the ring of polynomials in a single indeterminate $x$ over $\gf_q$. A polynomial
$f \in\gf_q[x]$ is called a {\em permutation polynomial} (PP) of $\gf_q$ if it induces
a one-to-one map from $\gf_q$ to itself.

Permutation polynomials over finite fields have been an interesting
subject of study for many years, and have applications in coding
theory \cite{DH13, LC07}, cryptography \cite{RSA, SH}, combinatorial
designs \cite{DY06}, and other
areas of mathematics and engineering. Information about properties,
constructions, and applications of permutation polynomials may be
found in Lidl and Niederreiter \cite{LR97}, and Mullen \cite{Mull}.
Some recent progress on permutation
polynomials can be found in \cite{AGW, DQWYY15, CK09, CK08, DXY, DY06, GS16, Hou151, Hou152, LP97, LHT13, LQC15, Ky10, Ma11, MZFG15,  PL01, Wa07, YD, ZZH, ZH, Zi09, Zi10, Zi13}.

Permutation binomials and trinomials have attracted people's attention due to their simple algebraic form and
additional extraordinary properties. Only a very small number of classes of permutation binomials and trinomials are known.
In this paper, we are particularly interested in classes of permutation trinomials over finite fields
with even characteristic. It is known that there are no permutation binomials with
both nonzero coefficients equal to 1 over finite fields with even characteristic. This
motivates us to find new classes of permutation trinomials with  coefficients
over finite fields with even characteristic. For a brief
survey on known classes of permutation trinomials, we refer the reader to \cite{DQWYY15, GS16, Hou151, LQC15, YD, Zi13}. In \cite{GS16}, the authors obtained some new types of
permutation trinomials and proposed the following conjecture.

\begin{conjecture} {\rm (\cite{GS16} Conjecture 2)} The polynomials $g(x) := x^5+x^{2^m+4}+x^{5\times2^m}\in\gf_{2^{2m}}[x]$ are permutation
trinomials over $\gf_{2^{2m}}$ if and only if $m\equiv2\pmod{4}$\end{conjecture}

In this paper, using some known results for permutation polynomials of the form $x^rh\left(x^{(q-1)/s}\right)$ in Zieve \cite{Zi13}, we prove the above conjecture (see, Theorem 4.2).

There are also many permutation trinomials in \cite{Zi13} by Zieve, which are not so well-known because they are not explicitly presented.  Another motivation of this paper is to give the explicit form of all  permutation trinomials over $\gf_{2^m}$ in Zieve's paper \cite{Zi13}.

This paper is organized as follows. In Section 2, we introduce some basic notation and lemmas. In particular, we introduce the
 QM equivalence of polynomials over finite fields (see, Definition 2.2).   In Section 3, we determine all permutation
trinomials over $\gf_{2^m}$ in Zieve's paper \cite{Zi13}. We prove the above conjecture and obtain some new  permutation trinomials over $\gf_{2^m}$ in Section 4. In Section 5, we show that some classes of permutation trinomials with parameters in \cite{MZFG15} and \cite{LQC15} are QM equivalent to some known permutation trinomials.

Notation: $\gf_q^\ast=\gf_q\setminus\{0\}$, $\bar{\gf}_q$ denotes the algebraic closure of $\gf_q$ and $\mu_{q+1}$ denotes the set of $(q+1)$-th roots of unity in $\gf_{q^2}^\ast$.

\section{Auxiliary results \& the main Lemma}

In this section, we present some auxiliary results that will be needed in the
sequel.
%For two positive integers $m$ and $n$ with $m|n$, we use $Tr_m^n(\cdot)$ to denote the $trace  \,\, function $ from $\gf_{2^n}$ to $\gf_{2^m}$, i.e.,
%$$ Tr_m^n(x)=x+x^{2^m}+x^{2^{2m}}+\cdots+x^{2^{(n/m-1)m}}.$$
%For $m=1$, we get the absolute trace function mapping onto the prime field $\gf_2$, which is denoted by $Tr_n$.
Let $m$ be a positive integer. For each element $x$ in the field $\gf_{2^{2m}}$, we define
$$\bar{x}=x^{2^m}$$
in analogy with the usual complex conjugate. Then we have

(i) $\overline{x+y}=\bar{x}+\bar{y}$ and $\overline{xy}=\bar{x}\bar{y}$ for all $x, y\in\gf_{2^{2m}}$, and

(ii) $x+\bar{x}\in \gf_{2^{m}}$ and  $x\bar{x}\in \gf_{2^{m}}$ for all $x\in \gf_{2^{2m}}$.

The following result is obvious, so we omit the proof.

\begin{lemma} Let $f(x)\in\gf_q[x]$ be a polynomial with $f(0)=0$. Then $f(x)$ is a PP  over $\gf_q$ if and only if $f(x)$ is a bijection from $\gf_q^\ast$ to $\gf_q^\ast$.\end{lemma}

If $f(x)$ is a bijection from $\gf_q^\ast$ to $\gf_q^\ast$, then we say that $f(x)$ is a PP over $\gf_q^\ast$ throughout the paper. By Lemma 2.1, we can focus our attention to the problem as to when $f(x)$ is a PP over $\gf_q^\ast$ whenever we study the permutation property of $f(x)$ over  $\gf_q$.

\begin{definition}Two permutation polynomials $f(x)$ and $g(x)$ in $\gf_q[x]$ are said to be quasi-multiplicative (QM, for short) equivalent if
there exists an integer $1 \le d \le q - 1$ with $\gcd(d,\, q- 1) = 1$ and   $f(x) \equiv ag(cx^d)\pmod{x^q-x}$ , where $a, \,\, c\in\gf_{q}^\ast$.\end{definition}

Obviously, $f(x)$ is a PP  over $\gf_q^\ast$ if and only if $af(cx^d)$, where $a, \,\, c\in\gf_{q}^\ast,\,\, \gcd(d,\,\, q-1)=1$, is a PP  over $\gf_q^\ast$. For the related references on the equivalence of polynomials, we refer the reader to \cite{Hou152} and \cite{LQC15}.

Zieve's constructions rely on the following result whose short proofs were given in \cite{Zi08, Zi09, Zi10, Zi13}.

\begin{lemma}\label{lezi} Pick an $h\in\gf_q[x]$ and integers  $r>0, \,\, s>0$  such that $s|(q-1)$. Then $f(x):=x^rh\left(x^{(q-1)/s}\right)$ permutes $\gf_q$ if and only if

(1) $\gcd(r,\,\, (q-1)/s)=1$ and

(2) $x^rh(x)^{(q-1)/s}$ permutes the set of $s$-th roots of unity in $\gf_q^\ast$. \end{lemma}

We also have the following result whose proof can be found in \cite{Zi08, Zi09, Zi10, Zi13}.

\begin{remark}\label{rezi} Pick an $h\in\gf_q[x]$ and integers $r, \,\, s>0$ such that $s|(q-1)$. Then $f(x):=x^rh\left(x^{(q-1)/s}\right)$ permutes $\gf_q^\ast$ if and only if

(1) $\gcd(r,\,\, (q-1)/s)=1$ and

(2) $x^rh(x)^{(q-1)/s}$ permutes the set of $s$-th roots of unity in $\gf_q^\ast$. \end{remark}

From Lemma \ref{lezi} and Remark \ref{rezi}, we obtain the following result easily.

\begin{proposition}\label{pozi} Pick an $h\in\gf_q[x]$ and integers $r>0, \,\, s>0$ such that $s|(q-1)$. Assume that $f(x):=x^rh\left(x^{(q-1)/s}\right)$ permutes $\gf_q$, then
 $x^rh(x)^{(q-1)/s}$ permutes the set of $s$-th roots of unity in $\gf_q^\ast$. Furthermore, let $k$ be an integer, then $g(x):=x^{ks}f(x)$ permutes $\gf_q^\ast$ if and only if
$\gcd(r+ks,\,\, (q-1)/s)=1$. \end{proposition}

\begin{proof} By Lemma \ref{lezi}, we deduce that $x^rh(x)^{(q-1)/s}$ permutes the set of $s$-th roots of unity in $\gf_q^\ast$. Since  for any $\alpha$ in the set of $s$-th roots of unity of $\gf_q^\ast$, $\alpha^{r+ks}h(\alpha)^{(q-1)/s}= \alpha^rh(\alpha)^{(q-1)/s}$,  by Remark \ref{rezi}, $g(x)$ permutes  $\gf_q^\ast$ if and only if  $\gcd(r+ks,\,\, (q-1)/s)=1$. This completes the proof.\end{proof}

We need the following proposition in Section 5, we refer the reader to \cite{MZFG15} and \cite{LQC15} for the original proof. For the convenience of the reader, we give a short proof here.
\begin{proposition}\label{potri1} Let $m>1$ be an odd integer, and write $k=\frac{m+1}{2}$. Then $f(x)=x+x^{2^k-1}+x^{2^k+1}$ is a permutation polynomial over $\gf_{2^m}^\ast$.
\end{proposition}

\begin{proof} Let $y=x^{2^k}$. Then we have $y^{2^{k-1}}=x$ and  $y^{2^k}=x^2$.
To prove that $f(x)$ is a permutation polynomial over $\gf_{2^m}^\ast$,  it suffices to prove that $f(x)=c$ has precisely one solution in $\gf_{2^m}^\ast$. Let $d=c^{2^k}$. Considering
the following equation
\begin{equation}\label{eq2-1} x+\frac{y}{x}+xy+c=0, \end{equation}
we then  have \begin{equation}\label{eq2-2} x^2+y+x^2y+cx=0.\end{equation}
Raising both sides of Eq.(\ref{eq2-2}) to the $2^{k-1}$-th power, we get
\begin{equation}\label{eq2-3}y+x+xy+c^{2^{k-1}}x^{2^{k-1}}=0.\end{equation}
Adding Eq.(\ref{eq2-1}) and Eq.(\ref{eq2-3}) together, we obtain
\begin{equation}\label{eq2-4}
y+\frac{y}{x}+c+c^{2^{k-1}}x^{2^{k-1}} =0.\end{equation}
Raising both sides of Eq.(\ref{eq2-4}) to the $2^k$-th power, we get
\begin{equation}x^2+\frac{x^2}{y}+d+cx=0.\end{equation}
Multiplying both sides of the above equation by $(x+c)\frac{y}{x}$, we obtain
\begin{equation}\label{eq2-6}x^2y+c^2y+dy+cd\frac{y}{x}+x^2+cx=0. \end{equation}
Adding Eq.(\ref{eq2-2}) and Eq.(\ref{eq2-6}) together and simplifying the newly obtained equation, we get
\begin{equation}\label{eq2-7}
(1+c^2+d)x=cd.\end{equation}
We claim that $1+c^2+d\ne0$ for any $c\in\gf_{2^m}^\ast$. Otherwise, we would have $1+c^2+d=0$, raising both sides to the $2^k$-th power, we get $1+d^2+c^2=0$.  It follows that $d=d^2$, i.e., $d=0$ or $1$ and $c=0$ or $1$, which is impossible.  Since $1+c^2+d \neq 0, \,\, cd\in \gf_{2^m}^\ast$, we know that Eq.(\ref{eq2-7}) has precisely one solution in  $\gf_{2^m}^\ast$. Then $f(x)=c$ has precisely one solution in  $\gf_{2^m}^\ast$ for any $c\in\gf_{2^m}^\ast$. Therefore $f(x)$ is a permutation polynomial over $\gf_{2^m}^\ast$. This completes the proof.\end{proof}

The following two results are needed in Section 3.
\begin{lemma}\label{lelucas} {\rm (Lucas formula)} Let  $p$ be a prime, and
$$m=\sum_{i=0}^lm_ip^i \quad \mbox{and}\quad k=\sum_{i=0}^lk_ip^i$$
be representations of $m$ and $k$ to the basis $p$, that is, $0\le m_i, \,\, k_i<p$. Then
$${m\choose k}=\prod_{i=0}^l {m_i\choose k_i} \pmod{p}.$$\end{lemma}

\begin{remark}\label{relucas}From the above Lucas formula, we can easily derive that ${m\choose k}\not\equiv0\pmod{p}$ if and only if $k_i\le m_i$ for all $i=0, 1, \ldots, l$. Hence there are precisely $\prod_{i=0}^l(m_i+1)$ integers $k$ such that $0\le k\le m$ and ${m\choose k}\not\equiv0\pmod{p}$.\end{remark}

\section{Zieve's permutation trinomials over $\gf_{2^m}$}

In this section, we determine all permutation trinomials over $\gf_{2^m}$ in Zieve's paper \cite{Zi13}.

\begin{lemma}\label{leZi1}{\rm(see, \cite{Zi13} Lemma 2.1)} Let  $q$ be a prime power, and let $l(x)\in\bar{\mathbb{F}}_q(x)$ be a degree-
one rational function. Then $l$ induces a bijection on $\mu_{q+1}$ if and only
if $l(x)$ equals either

(i)  $\beta/x$ with $\beta\in\mu_{q+1}$, or

(ii) $ (x -\gamma^q\beta)/(\gamma x-\beta)$  with  $\beta\in\mu_{q+1}$ and $\gamma\in\gf_{q^2}\setminus \mu_{q+1}$.\end{lemma}

\begin{lemma}\label{leZi2}{\rm(see,  \cite{Zi13} Lemma 3.1)} Let  $q$ be a prime power, and let $l(x)\in\bar{\mathbb{F}}_q(x)$ be a degree-
one rational function. Then $l$ induces a bijection from $\mu_{q+1}$ to $\gf_{q}\cup\{\infty\} $ if and only if $l(x) = (\delta x -\beta\delta^q)/( x-\beta)$  with  $\beta\in\mu_{q+1}$ and $\delta\in\gf_{q^2}\backslash \gf_{q}$.\end{lemma}

The following result is a little different from Zieve's Theorem 1.1 of \cite{Zi13}. The proof is almost the same as in  Theorem 1.1 of \cite{Zi13}. 

\begin{theorem}\label{thZiev1} Let $q$ be a prime power, let $n$ and $k$ be
integers, and let $\beta, \gamma \in \gf_{q^2}$ with $\beta^{q+1} = 1$ and $\gamma^{q+1}\ne 1$. Then
\begin{equation}\label{eqZi1} f(x) = x^{n+k(q+1)}  \left( (\gamma x^{q-1}-\beta)^n -\gamma (x^{q-1}-\gamma^q \beta)^n \right) \end{equation}
permutes $\gf_{q^2}^\ast$ if and only if $\gcd(n+2k,\,\,q-1) = 1$ and $\gcd(n,\,\, q+1) = 1$.\end{theorem}
\begin{proof} \noindent
 (by Michael E. Zieve)  Let  $N$  be a positive integer with  $N\equiv n \pmod{q^2-1}$ and $N+k(q-1)>0$. Since $\beta^{q+1} = 1$ and $\gamma^{q+1}\ne 1$, then  $\gamma x^{q-1} - \beta$  and  $x^{q-1} - \gamma^q \beta$  have no roots in $\gf_{q^2}^*$, so that $f(x)$ induces the same function on $\gf_{q^2}^*$ as does
$$F(x):= x^{N+k(q+1)} \left( (\gamma x^{q-1}-\beta)^N -\gamma (x^{q-1}-\gamma^q \beta)^N \right),$$ here $F(x)$ is obtained from $f(x)$ by replacing $n$ by $N$. Observe that $F(x)$ is a polynomial from $\gf_{q^2}$ to $\gf_{q^2}$, and $f(\alpha)=F(\alpha), \alpha\in\gf_{q^2}^\ast$.
Since $F(0)=0$, it follows that  $f(x)$  permutes $F_{q^2}^*$  if and only if  $F(x)$  permutes  $F_{q^2}$,
which follows from  Theorem 1.1 of Zieve \cite{Zi13},  $\gcd(N+2k,q-1)=\gcd(n+2k, q-1)=1$ and $\gcd(N, q+1)=\gcd(n, q+1)=1$ since $ N\equiv n \pmod{q^2-1}$.

\end{proof}

\begin{remark} Since $\beta^{q+1}=1$, we have $\beta=\beta_1^{q-1}$ for some $\beta_1\in\gf_{q^2}$, so $f(x)$ ( when $n+k(q+1)>0$ and $n>0$) is QM equivalent to the polynomial
$$\frac{f(\beta_1x)}{\beta_1^{n(q-1)+n+k(q+1)}}=x^{n+k(q+1)}  \left( (\gamma x^{q-1}-1)^n -\gamma (x^{q-1}-\gamma^q )^n \right).$$\end{remark}

Similarly, using Lemma 3.2 and Remark 2.4, we have the following theorem, which  is only slightly different from Theorem 1.2 of Zieve \cite{Zi13}.

\begin{theorem}\label{thZiev2} Let $q$ be a prime power, let $n>0$ and $k$ be
integers, and let $\beta, \delta \in \gf_{q^2}$ with $\beta^{q+1} = 1$ and $\delta\not\in\gf_{q}$. Then
\begin{equation}\label{eqZi2} f(x) = x^{n+k(q+1)}  \left( (\delta x^{q-1}-\beta\delta^q)^n -\delta (x^{q-1}- \beta)^n \right) \end{equation}
permutes $\gf_{q^2}^\ast$ if and only if $\gcd(n(n+2k),\,\,q-1) = 1$.\end{theorem}
\begin{proof}(By Michael E. Zieve) Let  $K$  be a positive integer with  $K\equiv k \pmod{q^2-1}$.  Then  $f(x)$  induces the same function on  $F_{q^2}^*$  as does
$$F(x) := x^{n+K(q+1)} \left( (\delta x^{q-1}-\beta\delta^q)^n -\delta (x^{q-1}- \beta)^n \right) ,$$
here $F(x)$ is obtained from $f(x)$ by replacing $k$ by $K$. Since $n>0$, then $f(\alpha)=F(\alpha), \alpha\in\gf_{q^2}^\ast$.
Since $F(0)=0$, it follows that  $f(x)$  permutes $F_{q^2}^*$  if and only if  $F(x)$  permutes  $F_{q^2}$,
which follows directly from  Theorem 1.2 of Zieve \cite{Zi13} and $\gcd(n+2K,q-1)=\gcd(n+2k, q-1)=1$ since $ K\equiv k \pmod{q^2-1}$.\end{proof}

Similarly, $f(x)$ (when $n+k(q+1)>0$) is QM equivalent to the polynomial
$$\frac{f(\beta_1x)}{\beta_1^{n(q-1)+n+k(q+1)}}=x^{n+k(q+1)}  \left( (\delta x^{q-1}-\delta^q)^n -\delta (x^{q-1}-1 )^n \right).$$

%\begin{theorem}\label{thZiev3} Let $q$ be a prime power, let $n$ and $k$ be
%integers with $n+k(q+1)>0$, and let $\beta, \gamma \in \gf_{q^2}$ satisfy $\beta^{q+1} = 1$ and $\gamma^{q+1}\ne 1$. Then
%\begin{equation}\label{eqZi3} f(x) = x^{k(q+1)+n}  \left( (\gamma -\beta x^{q-1})^n -\gamma (1-\gamma^q \beta x^{q-1})^n \right) \end{equation}
%permutes $\gf_{q^2}$ if and only if $\gcd(2k+n,\,\,q-1) = 1$ and $\gcd(n,\,\, q+1) = 1$.\end{theorem}
%\begin{proof}Since
%$$f_1(x)=x^{(n+2k)(q+1)}f(1/x)$$
%when $x\ne0$ and $f_1(0)=0$, so $f_1(x)$ permutes $\gf_{q^2}$ if and only if $\gcd(n+2k,\,\,q-1) = 1$ and $\gcd(n,\,\, q+1) = 1$.\end{proof}

%\begin{theorem}\label{thZiev4} Let $q$ be a prime power, let $n$ and $k$ be
%integers with $n+k(q+1)>0$, and let $\beta, \gamma \in \gf_{q^2}$ satisfy $\beta^{q+1} = 1$ and $\delta\not\in\gf_{q}$. Then
%\begin{equation}\label{eqZi4} f(x) = x^{k(q+1)+n}  \left( (\delta-\beta\delta^q x^{q-1})^n -\delta (1- \beta x^{q-1})^n \right) \end{equation}
%permutes $\gf_{q^2}$ if and only if $\gcd(n(2k+n),\,\,q-1) = 1$.\end{theorem}

\begin{remark} In Theorems 3.3 and 3.5, both $n$ and $k$ can be taken to be negative integers. Since $\alpha^n=\alpha^{q^2-1+n}, \, \alpha\in\gf_{q^2}^\ast$, without loss of generality, we can always assume that $n>0$. Hence we assume that $n>0$ throughout the paper.\end{remark}

\subsection{ The first case}
Let $p$ be a prime,  $n=\sum_{i=1}^ln_ip^i, \,\, 0\le n_i<p$
be the  representation of $n$ to the basis $p$. By Remark \ref{relucas}, there are precisely $\prod_{i=0}^l(n_i+1)$ integers $k$ with $0\le k\le n$ and ${ n \choose k}\not\equiv0\pmod{p}$. For an integer $k$ with $0\le k\le n$,  the coefficients of $x^{(n-k)(q-1)}$ and $x^{k(q-1)}$ in the polynomial $(\gamma x^{q-1}-\beta)^n -\gamma (x^{q-1}-\gamma^q \beta)^n$ are
$$(-1)^k\left(\gamma^{n-k}-\gamma^{1+kq}\right){n\choose k}\beta^k \quad \mbox{and}\quad (-1)^{n-k}\left(\gamma^{k}-\gamma^{1+(n-k)q}\right){n\choose k}\beta^{n-k}. $$

If ${ n \choose k}\not\equiv0\pmod{p}$ and both the coefficients of $x^{(n-k)(q-1)}$ and $x^{k(q-1)}$ in the polynomial $(\gamma x^{q-1}-\beta)^n -\gamma (x^{q-1}-\gamma^q \beta)^n$ are 0, then we have $\gamma^{n-k}=\gamma^{1+kq}$. Hence
$$\gamma^{k}-\gamma^{1+(n-k)q}=\gamma^{k}-\gamma \gamma^{(1+kq)q}=\gamma^{k}-\gamma^{k+1+q},$$
here we have used the fact that  $\gamma^{q^2}=\gamma$, which is 0 if and only if $\gamma^{q+1}=1$, a contradiction to the assumption of Theorem \ref{thZiev1}. Therefore we have proved that at least one of the coefficients of $x^{(n-k)(q-1)}$ and $x^{k(q-1)}$ in the polynomial $(\gamma x^{q-1}-\beta)^n -\gamma (x^{q-1}-\gamma^q \beta)^n$ is not 0 when ${ n \choose k}\not\equiv0\pmod{p}$. Consequently, the polynomial $(\gamma x^{q-1}-\beta)^n -\gamma (x^{q-1}-\gamma^q \beta)^n$ is a trinomial polynomial of $x$ only when $\prod_{i=0}^l(n_i+1)\le 6$,  it follows that $n=2^s+2^t, \,\, 0\le s<t$ for $p=2$.

Let $p=2$, $n=2^s+2^t, \,\, 0\le s<t$. Notice that
$$(\gamma x^{q-1}-1)^n -\gamma (x^{q-1}-\gamma^q )^n=(\gamma x^{q-1}+1)^n +\gamma (x^{q-1}+\gamma^q )^n $$
$$=(\gamma^n+\gamma)x^{n(q-1)}+(\gamma^{2^t}+\gamma^{1+2^sq}) x^{2^t(q-1)}+(\gamma^{2^s}+\gamma^{1+2^tq})x^{2^s(q-1)}+(1+\gamma^{1+nq}).$$
We consider the following four cases.

{\bf Case 1:} $\gamma^n+\gamma=0$. Then $\gamma^{2^t}=\gamma^{1-2^s}$, $\gamma^{2^s}=\gamma^{1-2^t}$ and
$$\gamma^{2^t}+\gamma^{1+2^sq}=\gamma^{2^t}\left(1+\gamma^{1-2^t+2^sq}\right)= \gamma^{2^t}\left(1+\gamma^{2^s(q+1)}\right), $$

$$\gamma^{2^s}+\gamma^{1+2^tq}=\gamma^{2^s}\left(1+\gamma^{2^t(q+1)}\right), \quad 1+\gamma^{1+nq}=1+\gamma^{q+1}.$$
In this case, we have
$$(\gamma x^{q-1}+1)^n +\gamma (x^{q-1}+\gamma^q )^n $$
$$=\gamma^{2^t}\left(1+\gamma^{2^s(q+1)}\right)x^{2^t(q-1)}+\gamma^{2^s}\left(1+\gamma^{2^t(q+1)}\right)x^{2^s(q-1)}+(1+\gamma^{q+1}).$$

Now we take $n=2^s+2^t,$ where $2\not|st$ and $m$ is even with $\gcd(n,\,\, 2^m+1)=1$. Take $\gamma=\omega$ to be a primitive cubic root of unity, then $\gamma^{2^t}\left(1+\gamma^{2^s(q+1)}\right)=\gamma^{2^s}\left(1+\gamma^{2^t(q+1)}\right)=1+\gamma^{q+1}=\omega$. Hence we have the following result.

\begin{corollary}Let $q=2^{m},\,\, 2|m $  and let $n=2^s+2^t, \,\, 0\le s< t, \,\,2\not|st$. Then
$$ f(x) = x^{n+k(q+1)}  \left(  x^{2^t(q-1)}+ x^{2^s(q-1)}+1 \right)  $$
permutes $\gf_{q^2}^\ast$ if and only if $\gcd(n+2k,\,\,q-1) = 1$ and $\gcd(n, \,\, 2^m+1)=1$. In particular,
$$g(x)=x^{10+k(q+1)}  \left(  x^{8(q-1)}+ x^{2(q-1)}+1 \right)$$
permutes $\gf_{q^2}^\ast$ if and only if $\gcd(10+2k,\,\,q-1) = 1$ and $4|m$.\end{corollary}

{\bf Case 2:} $\gamma^{2^t}+\gamma^{1+2^sq}=0$. Then $\gamma^{2^t}=\gamma^{1+2^sq}$,  and
$$\gamma^{2^s}+\gamma^{1+2^tq}=\gamma^{2^s}+\gamma\gamma^{(1+2^sq)q}= \gamma^{2^s}\left(1+\gamma^{q+1}\right), $$
$$\gamma+\gamma^{n}=\gamma+\gamma^{2^s+2^t}=\gamma+\gamma\cdot\gamma^{2^s(1+q)}=\gamma\left(1+\gamma^{2^s(q+1)}\right),$$
$$1+\gamma^{1+nq}=1+\gamma^{1+(2^s+2^t)q}=1+\gamma^{2^tq}\cdot\gamma^{1+2^sq}=1+\gamma^{2^t(q+1)}.$$

In this case, we have
$$(\gamma x^{q-1}+1)^n +\gamma (x^{q-1}+\gamma^q)^n $$
$$=\gamma\left(1+\gamma^{2^s(q+1)}\right)x^{n(q-1)}+\gamma^{2^s}\left(1+\gamma^{(q+1)}\right)x^{2^s(q-1)}+\left(1+\gamma^{2^t(q+1)}\right).$$

If we take $n=2^s+2^t,\,\, 2|s, \,\,2\not|t$ and $m$ is even, then $\gcd(n,\,\, 2^m+1)=1$. Take $\gamma=\omega$ to be a primitive cubic root of unity, then $\gamma\left(1+\gamma^{2^s(q+1)}\right)=\gamma^{2^s}\left(1+\gamma^{q+1}\right)=1+\gamma^{2^t(q+1)}=\omega^2$. Hence, we obtain the following result.

\begin{corollary}Let $q=2^{m},\,\, 2|m$  and let $n=2^s+2^t, \,\, 0\le s< t, \,\,2|s, \,\,2\not|t$. Then
$$ f(x) = x^{n+k(q+1)}  \left(  x^{n(q-1)}+ x^{2^s(q-1)}+1 \right)  $$
permutes $\gf_{q^2}^\ast$ if and only if $\gcd(n+2k,\,\,q-1) = 1$. In particular,
$$g(x)=x^{3+k(q+1)}  \left(  x^{3(q-1)}+ x^{q-1}+1 \right)$$
permutes $\gf_{q^2}^\ast$ if and only if $\gcd(3+2k,\,\,q-1) = 1$.\end{corollary}

{\bf Case 3:} $\gamma^{2^s}+\gamma^{1+2^tq}=0$. Then $\gamma^{2^s}=\gamma^{1+2^tq}$,  and
$$\gamma^{2^t}+\gamma^{1+2^sq}=\gamma^{2^t}+\gamma\gamma^{(1+2^tq)q}= \gamma^{2^t}\left(1+\gamma^{q+1}\right), $$
$$\gamma+\gamma^{n}=\gamma+\gamma^{2^s+2^t}=\gamma+\gamma\cdot\gamma^{2^t(1+q)}=\gamma\left(1+\gamma^{2^t(q+1)}\right),$$
$$1+\gamma^{1+nq}=1+\gamma^{1+(2^s+2^t)q}=1+\gamma^{2^sq}\cdot\gamma^{1+2^tq}=1+\gamma^{2^s(q+1)}.$$

In this case, we have
$$(\gamma x^{q-1}+1)^n +\gamma (x^{q-1}+\gamma^q)^n $$
$$=\gamma\left(1+\gamma^{2^t(q+1)}\right)x^{n(q-1)}+\gamma^{2^t}\left(1+\gamma^{q+1}\right)x^{2^t(q-1)}+\left(1+\gamma^{2^s(q+1)}\right).$$

We take $n=2^s+2^t,\,\, 2|t, \,\,2\not|s$ and $m$ is even, then $\gcd(n,\,\, 2^m+1)=1$. Take $\gamma=\omega$ to be a primitive cubic root of unity, then $\gamma\left(1+\gamma^{2^t(q+1)}\right)=\gamma^{2^t}\left(1+\gamma^{q+1}\right)=1+\gamma^{2^s(q+1)}=\omega^2$. Thus, we get the following.

\begin{corollary}Let $q=2^{m},\,\, 2|m$  and  let $n=2^s+2^t, \,\, 0\le s< t, \,\,2|t, \,\,2\not|s$. Then
$$ f(x) = x^{n+k(q+1)}  \left(  x^{n(q-1)}+ x^{2^t(q-1)}+1 \right)  $$
permutes $\gf_{q^2}^\ast$ if and only if $\gcd(n+2k,\,\,q-1) = 1$. In particular,
$$g(x)=x^{6+k(q+1)}  \left(  x^{6(q-1)}+ x^{4(q-1)}+1 \right)$$
permutes $\gf_{q^2}^\ast$ if and only if $\gcd(6+2k,\,\,q-1) = 1$.\end{corollary}

{\bf Case 4:} $1+\gamma^{1+nq}=0$. Then $\gamma^{-n}=\gamma^q$.

$$\gamma+\gamma^n=\gamma^n(1+\gamma^{1-n})=\gamma^n(1+\gamma^{1+q}),$$
$$\gamma^{2^t}+\gamma^{1+2^sq}=\gamma^{-2^tq}\left(\gamma^{2^t(q+1)}+\gamma^{(1+nq)}\right)= \gamma^{-2^tq}\left(1+\gamma^{2^t(q+1)}\right), $$
$$\gamma^{2^s}+\gamma^{1+2^tq}=\gamma^{-2^sq}\left(\gamma^{2^s(q+1)}+\gamma^{(1+nq)}\right)= \gamma^{-2^sq}\left(1+\gamma^{2^s(q+1)}\right), $$

In this case, we have
$$(\gamma x^{q-1}-1)^n -\gamma (x^{q-1}-\gamma^q )^n $$
$$=\gamma^n(1+\gamma^{1+q})x^{n(q-1)}+\gamma^{-2^tq}\left(1+\gamma^{2^t(q+1)}\right)x^{2^t(q-1)}+\gamma^{-2^sq}\left(1+\gamma^{2^s(q+1)}\right)x^{2^s(q-1)}.$$

 Take $n=2^s+2^t,\,\, 2|s, \,\,2|t$ and $m$ is even with $\gcd(n,\,\, 2^m+1)=1$. Take $\gamma=\omega$ to be a primitive cubic root of unity, then $\gamma^n(1+\gamma^{1+q})=\gamma^{-2^tq}\left(1+\gamma^{2^t(q+1)}\right)=\gamma^{-2^sq}\left(1+\gamma^{2^s(q+1)}\right)=1$. We have then obtained the following.

\begin{corollary}
Let $q=2^{m},\,\, 2|m$  and let $n=2^s+2^t, \,\, 0\le s< t, \,\,2|s, \,\,2|t$. Then
$$ f(x) = x^{n+k(q+1)}  \left(  x^{n(q-1)}+ x^{2^t(q-1)}+x^{2^s(q-1)} \right)  $$
permutes $\gf_{q^2}^\ast$ if and only if $\gcd(n+2k,\,\,q-1) = 1$ and $\gcd(n,\,\, q+1)=1$. In particular,
$$g(x)=x^{5+k(q+1)}  \left(  x^{5(q-1)}+ x^{4(q-1)}+x^{q-1} \right)$$
permutes $\gf_{q^2}^\ast$ if and only if $\gcd(5+2k,\,\,q-1) = 1$ and $4|m$.\end{corollary}

\subsection{The second case}

For an integer $k$ with $0\le k\le n$,  the coefficients of $x^{(n-k)(q-1)}$ and $x^{k(q-1)}$ in the polynomial $(\delta x^{q-1}-\beta\delta^q)^n -\delta (x^{q-1}- \beta)^n$ are
$$(-1)^k\left(\delta^{n-k+kq}-\delta\right){n\choose k}\beta^k \quad \mbox{and}\quad (-1)^{n-k}\left(\delta^{(n-k)q+k}-\delta\right){n\choose k}\beta^{n-k}. $$

If ${ n \choose k}\not\equiv0\pmod{p}$ and both the coefficients of $x^{(n-k)(q-1)}$ and $x^{k(q-1)}$ in the polynomial $(\delta x^{q-1}-\beta\delta^q)^n -\delta (x^{q-1}- \beta)^n$ are 0, then we have $\delta^{n-k+kq}=\delta$. Hence
$$\delta^{(n-k)q+k}-\delta=\delta^q-\delta,$$
here we have used the fact that $\delta^{q^2}=\delta$, which is 0 if and only if $\delta\in\gf_q$, a contradiction to the assumption of Theorem \ref{thZiev2}. Therefore we have proved that at least one of the coefficients of $x^{(n-k)(q-1)}$ and $x^{k(q-1)}$ in the polynomial $(\delta x^{q-1}-\beta\delta^q)^n -\delta (x^{q-1}- \beta)^n$ is not 0 when ${ n \choose k}\not\equiv0\pmod{p}$. Consequently, the polynomial $(\delta x^{q-1}-\beta\delta^q)^n -\delta (x^{q-1}- \beta)^n$ is a trinomial polynomial of $x$ only when $\prod_{i=0}^l(n_i+1)\le 6$,  it follows that $n=2^s+2^t, \,\, 0\le s<t$ for $p=2$.

Let $p=2$, $n=2^s+2^t, \,\, 0\le s<t$. Observe that
$$(\delta x^{q-1}-\delta^q)^n -\delta (x^{q-1}- 1)^n=(\delta x^{q-1}+\delta^q)^n +\delta (x^{q-1}+ 1)^n  $$
$$=(\delta+\delta^n)x^{n(q-1)}+(\delta+\delta^{2^t+2^sq})x^{2^t(q-1)}+(\delta+\delta^{2^s+2^tq})x^{2^s(q-1)}+(\delta+\delta^{nq}).$$
We also have the following four cases for consideration.

{\bf Case 1:} $ \delta+\delta^n=0$. Similarly, we have
%Then $\delta^{2^t}=\delta^{1-2^s}$, $\delta^{2^s}=\delta^{1-2^t}$ and
%$$\delta+\delta^{2^t+2^sq}=\delta\left(1+\delta^{2^s(q-1)}\right),$$
%$$\delta+\delta^{2^s+2^tq}=\delta\left(1+\delta^{2^t(q-1)}\right),$$
%$$\delta+\delta^{nq}=\delta+\delta^{q}=\delta(1+\delta^{q-1}).$$
$$(\delta x^{q-1}+\delta^q)^n +\delta (x^{q-1}+1)^n $$
$$=\delta\left(1+\delta^{2^s(q-1)}\right)x^{2^t(q-1)}+\delta\left(1+\delta^{2^t(q-1)}\right)x^{2^s(q-1)}+\delta(1+\delta^{q-1}).$$

Now if we take $n=2^s+2^t,\,\, 2\not|st$ and $m$ is odd, then $\gcd(n, 2^m-1)=1$. Take $\delta=\omega$ to be a primitive cubic root of unity, then $\delta\left(1+\delta^{2^s(q-1)}\right)=\delta\left(1+\delta^{2^t(q-1)}\right)=\omega^2$ and $\delta(1+\delta^{q-1})=1$. Hence
$$(\delta (\omega^2 x)^{q-1}+\delta^q)^n +\delta ((\omega^2 x)^{q-1}+1)^n= x^{2^t(q-1)}+x^{2^s(q-1)}+1, $$
and we
have the following.

\begin{corollary}Let $q=2^{m},\,\, 2\not|m$  and  let $n=2^s+2^t, \,\, 0\le s< t, \,\,2\not|st$. Then
$$f(\omega^2x)/\omega^2 = x^{n+k(q+1)}  \left(  x^{2^t(q-1)}+x^{2^s(q-1)}+1 \right)  $$
permutes $\gf_{q^2}^\ast$ if and only if $\gcd(n+2k,\,\,q-1) = 1$. In particular,
$$g(x)=x^{10+k(q+1)}  \left(  x^{8(q-1)}+ x^{2(q-1)}+1 \right)$$
permutes $\gf_{q^2}^\ast$ if and only if $\gcd(10+2k,\,\,q-1) = 1$.\end{corollary}

{\bf Case 2:} $ \delta+\delta^{2^t+2^sq}=0$. Similarly,
%Then we have
%$$\delta+\delta^n=\delta^n\left(1+\delta^{2^s(q-1)}\right),$$
%$$\delta+\delta^{2^s+2^tq}=\delta+\delta^q=\delta\left(1+\delta^{q-1}\right),$$
%$$\delta+\delta^{nq}=\delta(1+\delta^{2^t(q-1)}).$$
$$(\delta x^{q-1}+\delta^q)^n +\delta (x^{q-1}+ 1)^n $$
$$=\delta^n\left(1+\delta^{2^s(q-1)}\right)x^{n(q-1)}+\delta\left(1+\delta^{q-1}\right)x^{2^s(q-1)}+\delta(1+\delta^{2^t(q-1)}).$$

If we take $n=2^s+2^t,\,\, 2|s, \,\,2\not|t$ and $m$ is odd, then $\gcd(n, 2^m-1)=1$. Take $\delta=\omega$ to be a primitive cubic root of unity, then $\delta^{n}\left(1+\delta^{2^s(q-1)}\right)=\delta(1+\delta^{2^t(q-1)})=\omega^2$ and $\delta\left(1+\delta^{q-1}\right)=1$,  and we have
$$(\delta (\omega^2 x)^{q-1}+\delta^q)^n +\delta ((\omega^2 x)^{q-1}+ 1)^n =\omega^2\left(x^{n(q-1)}+x^{2^s(q-1)}+1\right).$$
Hence we have proved the following corollary.

\begin{corollary}Let $q=2^{m},\,\, 2\not|m$  and  let $n=2^s+2^t, \,\, 0\le s< t, \,\,2|s, \,\,2\not|t$. Then
$$f(\omega^2 x)/\omega^2 = x^{n+k(q+1)}  \left(x^{n(q-1)}+x^{2^s(q-1)}+1 \right)  $$
permutes $\gf_{q^2}^\ast$ if and only if $\gcd(n+2k,\,\,q-1) = 1$. In particular,
$$g(x)=x^{3+k(q+1)}  \left(  x^{3(q-1)}+ x^{q-1}+1 \right)$$
permutes $\gf_{q^2}^\ast$ if and only if $\gcd(3+2k,\,\,q-1) = 1$.\end{corollary}

{\bf Case 3:} $ \delta+\delta^{2^s+2^tq}=0$. Similarly, %Then we have
%$$\delta+\delta^n=\delta^n\left(1+\delta^{2^t(q-1)}\right),$$
%$$\delta+\delta^{2^t+2^sq}=\delta+\delta^q=\delta\left(1+\delta^{q-1}\right),$$
%$$\delta+\delta^{nq}=\delta(1+\delta^{2^s(q-1)}).$$
$$(\delta x^{q-1}-\delta^q)^n -\delta (x^{q-1}- 1)^n $$
$$=\delta^n\left(1+\delta^{2^t(q-1)}\right)x^{n(q-1)}+\delta\left(1+\delta^{q-1}\right)x^{2^t(q-1)}+\delta(1+\delta^{2^s(q-1)}).$$

We take $n=2^s+2^t,\,\, 2\not|s, \,\,2|t$ and $m$ is odd, then $\gcd(n,\,\, 2^m-1)=1$. Take $\delta=\omega$ to be a primitive cubic root of unity, then $\delta^{n}\left(1+\delta^{2^t(q-1)}\right)=\delta(1+\delta^{2^s(q-1)})=\omega^2$ and $\delta\left(1+\delta^{q-1}\right)=1$,  and we have
$$(\delta (\omega^2 x)^{q-1}+\delta^q)^n +\delta ((\omega^2 x)^{q-1}+ 1)^n =\omega^2\left(x^{n(q-1)}+x^{2^t(q-1)}+1\right).$$
Therefore we have proved the following corollary.

\begin{corollary}Let $q=2^{m},\,\, 2\not|m$  and  let $n=2^s+2^t, \,\, 0\le s< t, \,\,2\not|s, \,\,2|t$. Then
$$f(\omega^2x)/\omega^2 = x^{n+k(q+1)}  \left(x^{n(q-1)}+x^{2^t(q-1)}+1 \right)  $$
permutes $\gf_{q^2}^\ast$ if and only if $\gcd(n+2k,\,\,q-1) = 1$. In particular,
$$g(x)=x^{6+k(q+1)}  \left(  x^{6(q-1)}+ x^{4(q-1)}+1 \right)$$
permutes $\gf_{q^2}^\ast$ if and only if $\gcd(6+2k,\,\,q-1) = 1$.\end{corollary}

{\bf Case 4:} $ \delta+\delta^{nq}=0$. Similarly, % Then $\delta^{2^tq}=\delta^{1-2^sq}$, $\delta^{2^sq}=\delta^{1-2^tq}$ and
%$$\delta+\delta^{2^t+2^sq}=\delta^{2^t+2^sq}\left(1+\delta^{2^t(q-1)}\right),$$
%$$\delta+\delta^{2^s+2^tq}=\delta^{2^s+2^tq}\left(1+\delta^{2^s(q-1)}\right),$$
%$$\delta+\delta^{n}=\delta+\delta^{q}=\delta(1+\delta^{q-1}).$$
$$(\delta x^{q-1}+\delta^q)^n +\delta (x^{q-1}+1)^n $$
$$=\delta(1+\delta^{q-1})x^{n(q-1)}+\delta^{2^t+2^sq}\left(1+\delta^{2^t(q-1)}\right)x^{2^t(q-1)}+\delta^{2^s+2^tq}\left(1+\delta^{2^s(q-1)}\right)x^{2^s(q-1)}.$$

Take $n=2^s+2^t,\,\, 2|s, 2|t$ and $m$ is odd, then $\gcd(n,\,\, 2^m-1)=1$. Take $\delta=\omega$ to be a primitive cubic root of unity, then $\delta(1+\delta^{q-1})=1$ and $\delta^{2^t+2^sq}\left(1+\delta^{2^t(q-1)}\right)=\delta^{2^s+2^tq}\left(1+\delta^{2^s(q-1)}\right)=\omega^2$, and we have
$$(\delta (\omega x)^{q-1}+\delta^q)^n +\delta ((\omega x)^{q-1}+ 1)^n =\omega^2\left( x^{n(q-1)}+x^{2^t(q-1)}+ x^{2^s(q-1)}\right).  $$
 Hence we get the next result.

\begin{corollary}Let $q=2^{m},\,\, 2\not|m$  and  let $n=2^s+2^t, \,\, 0\le s< t, \,\,2|s, \, 2|t$. Then
$$ f(\omega x)/\omega = x^{n+k(q+1)}  \left(  x^{n(q-1)}+x^{2^t(q-1)}+ x^{2^s(q-1)}\right)  $$
permutes $\gf_{q^2}^\ast$ if and only if $\gcd(n+2k,\,\,q-1) = 1$. In particular,
$$g(x)=x^{5+k(q+1)}  \left(  x^{5(q-1)}+ x^{4(q-1)}+x^{(q-1)} \right)$$
permutes $\gf_{q^2}^\ast$ if and only if $\gcd(5+2k,\,\,q-1) = 1$.\end{corollary}

Combining Corollaries 3.8, 3.9, 3.12 and 3.13, we obtain the following theorem.

\begin{theorem}\label{zifin} Let $m>0$ and $k$ be integers, $q=2^m$. Then $$f(x)=x^{3+k(q+1)}  \left(  x^{3(q-1)}+ x^{q-1}+1 \right)$$
permutes $\gf_{q^2}^\ast$ if and only if $\gcd(3+2k,\,\,q-1) = 1$. The polynomial
$$f(x)=x^{6+k(q+1)}  \left(  x^{6(q-1)}+ x^{4(q-1)}+1 \right)$$
permutes $\gf_{q^2}^\ast$ if and only if $\gcd(6+2k,\,\,q-1) = 1$. \end{theorem}

Take $k=-1,\,\, -2 $ in Theorem \ref{zifin}. Since $\alpha^{q^2-1}=1$ for any $\alpha\in\gf_{q^2}^\ast$, we obtain that
$$\alpha^{3-(q+1)}  \left(  \alpha^{3(q-1)}+ \alpha^{q-1}+1 \right)=\alpha^{q^2-1+3-(q+1)}  \left(  \alpha^{3(q-1)}+ \alpha^{q-1}+1 \right)=\alpha^q+\alpha^{2q-1}+ x^{q^2-q+1},$$
$$ \alpha^{3-2(q+1)}  \left( \alpha^{3(q-1)}+ \alpha^{q-1}+1 \right)=\alpha^{q^2-1+3-2(q+1)}  \left(  \alpha^{3(q-1)}+ \alpha^{q-1}+1 \right)=\alpha^{q-2} +  \alpha^{q^2-q-1}+ \alpha^{q^2-2q}, $$
$$\left(\alpha^q+\alpha^{2q-1}+ x^{q^2-q+1}\right)^q=\alpha+\alpha^{2q-1}+\alpha^{q^2-q+1}.$$
\begin{example} Let $m$ be a positive integer and $q=2^m$. Then the following polynomials
$$ f_1(x)=x+x^{2^{m+1}-1}+x^{2^{2m}-2^m+1}\quad \mbox{ and} \quad f_2(x)=x^{2^m-2} +  x^{2^{2m}-2^{m+1}}+ x^{2^{2m}-2^m-1}$$permute $\gf_{q^2}$.\end{example}
Let $b\in\gf_{q^2}^\ast$ and $a=b^{2(q-1)}$, then the order of $a$ divides $2^m+1$ and
$$g(x)=\frac{f_1(bx)}{b}=x+ax^{2^{m+1}-1}+a^{2^{m-1}}x^{2^{2m}-2^m+1}.$$
That is,  $g(x)$ is QM equivalent to $x+x^{2^{m+1}-1}+x^{2^{2m}-2^m+1} $.  Therefore we have proved that $x+ax^{2^{m+1}-1}+a^{2^{m-1}}x^{2^{2m}-2^m+1}$, where $a\in\gf_{2^{2m}}$ is any element of order dividing $2^m+1$, permutes $\gf_{2^{2m}}$.
This improves Theorem 4.9 in \cite{LQC15}.

\begin{remark} With the arguments in this section, we can obtain all permutation trinomials over $\gf_{2^m}$ in Zieve's paper \cite{Zi13}. We can also obtain many explicit permutation trinomials over $\gf_{2^m}$  from the preceding eight Corollaries. \end{remark}

\section{Proof of Conjecture 1.1}

In this section, we first prove Conjecture 1.1. Then we derive more explicit permutation trinomials by using Proposition \ref{pozi}.
By Proposition \ref{pozi}, Theorem \ref{zifin}, and  Theorems 3.1 and 3.4 in \cite{GS16},  we are ready to prove the following result.

\begin{lemma}\label{le41} Let $m>0$ be a positive integer and $q=2^m$. Then the following hold:

(i) The functions $\frac{x^3+x+1}{x^3+x^2+1}$, $\frac{x^3+x^2+1}{x^3+x+1}$,\, $\frac{x^6+x^2+1}{x^6+x^4+1}$ and $\frac{x^6+x^4+1}{x^6+x^2+1}$ permute $\mu_{q+1}$ in $\gf_{q^2}$.

(ii) The functions $\frac{x^4+x^3+x}{x^3+x+1}$ and $\frac{x^3+x+1}{x^4+x^3+x}$ permute $\mu_{q+1}$ in $\gf_{q^2}$ if and only if $\gcd(m,\,\, 3)=1$.

(iii) If $m$ is odd, then the  functions $\frac{x^5+x^4+x}{x^4+x+1}$ and $\frac{x^4+x+1}{x^5+x^4+x}$ permute $\mu_{q+1}$ in $\gf_{q^2}$\end{lemma}

\begin{proof}(i) By Theorem \ref{zifin}, $f(x)=x^{3-(q+1)}  \left(  x^{3(q-1)}+ x^{q-1}+1 \right)$
permutes $\gf_{q^2}^\ast$ since $\gcd(3-2,\,\,q-1) = 1$. Since
$$g(\alpha)=\alpha^{3-(q+1)}(\alpha^3+\alpha+1)^{q-1}=\frac{\alpha^3+\alpha^2+1}{\alpha^3+\alpha+1},$$
here we have used the fact that $\alpha\in\mu_{q+1}$, i.e., $\alpha^{q+1}=1$, we deduce that $\frac{x^3+x^2+1}{x^3+x+1}$  permutes $\mu_{q+1}$ in $\gf_{q^2}$ by Proposition \ref{pozi}. Note that
$$\frac{x^3+x+1}{x^3+x^2+1}=\left(\frac{x^3+x^2+1}{x^3+x+1}\right)^{-1}, $$
$$ \frac{x^6+x^4+1}{x^6+x^2+1}=\left(\frac{x^3+x^2+1}{x^3+x+1}\right)^{2}, \quad \frac{x^6+x^2+1}{x^6+x^4+1}=\left(\frac{x^3+x^2+1}{x^3+x+1}\right)^{-2},$$
so $\frac{x^3+x^2+1}{x^3+x+1}$, $\frac{x^6+x^2+1}{x^6+x^4+1}$ and $\frac{x^6+x^4+1}{x^6+x^2+1}$  permute $\mu_{q+1}$.

 (ii) By  Theorem 3.1 of \cite{GS16}, the polynomial $f_1(x): =x^4+x^{2^m+3}+x^{3\cdot2^m+1}\in\gf_{2^{2m}}[x]$ is a PP over $\gf_{2^{2m}}$ if and only if $\gcd(3, m)=1$. Since
 $$g(\alpha)=\alpha^4(\alpha^3+\alpha+1)^{q-1}=\frac{\alpha^4+\alpha^3+\alpha}{\alpha^3+\alpha+1},\quad \frac{x^3+x+1}{x^4+x^3+x}=\left(\frac{x^4+x^3+x}{x^3+x+1}\right)^{-1}, $$
 so by Proposition \ref{pozi}, $\frac{x^4+x^3+x}{x^3+x+1}$ and $\frac{x^3+x+1}{x^4+x^3+x}$ permute $\mu_{q+1}$ in $\gf_{q^2}$ if and only if $\gcd(m,\,\, 3)=1$.

 (iii) The conclusion follows from  Theorem 3.4 of \cite{GS16} and the same argument as in the proof of (ii). This completes the proof.
 \end{proof}

Now we will use the lemma above and Proposition \ref{pozi} to prove Conjecture 1.1.

\begin{theorem}
Let $m>0$  and $k$ be integers, $q=2^m$. Then
$$f(x) := x^{5+k(q+1)}\left(1+x^{2^m-1}+x^{5\cdot(2^m-1)}\right)\in\gf_{2^{2m}}[x]$$ permutes $\gf_{2^{2m}}^\ast$ if and only if $\gcd(5+2k, 2^m-1)=1$ and $2|m$. In particular,  $$g(x) := x^5+x^{2^m+4}+x^{5\times2^m}\in\gf_{2^{2m}}[x]$$ is a permutation
trinomial over $\gf_{2^{2m}}$ if and only if $m\equiv2\pmod{4}$.\end{theorem}

\begin{proof} The trinomial
$f(x) := x^{5+k(q+1)}(1+x^{2^m-1}+x^{5\cdot(2^m-1)})$ can be written as $$f(x)=x^{5+k(q+1)}h(x^{2^m-1}),$$
where $h(x):=1+x+x^5\in\gf_{2^{2m}}[x]$. By Remark \ref{rezi}, $f(x)$ permutes $\gf_{2^{2m}}^\ast$ if and only if $\gcd(5+k(q+1), q-1)=\gcd(5+2k, q-1)=1$ and $g(x):=x^{5+k(q+1)}h(x)^{2^m-1}$ permutes $\mu_{q+1}$.

If $m$ is odd, then $3|(q+1)$. Take $x=\omega$ to be the primitive cubic root of unity, then $\omega\in \mu_{q+1}$ and $g(\omega)=0$. If $2|m$, then $\gcd(3, q+1)=1$, which implies that $x^2+x+1\ne0$ for any $x\in\mu_{q+1}$. For $\alpha\in\mu_{q+1}$,
\begin{eqnarray*}g(\alpha)&=&\alpha^{5+k(q+1)}(1+\alpha+\alpha^5)^{q-1}\\
&=&\alpha^{5}\frac{(1+\alpha+\alpha^5)^{2^m}}{1+\alpha+\alpha^5}\\
&=&\alpha^{5}\frac{1+\alpha^{-1}+\alpha^{-5}}{1+\alpha+\alpha^5}\\
&=&\frac{1+\alpha^4+\alpha^5}{1+\alpha+\alpha^5}\\
&=&\frac{(1+\alpha+\alpha^2)(1+\alpha+\alpha^3)}{(1+\alpha+\alpha^2)(1+\alpha^2+\alpha^3)}\\
&=&\frac{1+\alpha+\alpha^3}{1+\alpha^2+\alpha^3}.\end{eqnarray*}
Since $G(x):=\frac{x^3+x+1}{x^3+x^2+1}$ permutes $\mu_{q+1}$ in $\gf_{q^2}$ by Lemma \ref{le41} (i), so by Proposition \ref{pozi} and Lemma 2.1, $f(x)$ permutes $\gf_{2^{2m}}^\ast$ if and only if $\gcd(5+2k,\,\, 2^m-1)=1$ and $2|m$. Since $\gcd(5,\,\, 2^m-1)=1$ only when $4\not|m$,  $g(x) := x^5+x^{2^m+4}+x^{5\times2^m}\in\gf_{2^{2m}}[x]$ is a permutation
trinomial over $\gf_{2^{2m}}$ if and only if $m\equiv2\pmod{4}$. This completes the proof.
\end{proof}

Let $k=-2$, then $\gcd(5-2\times2,\,\, 2^m-1)=1$. For any $\alpha\in\gf_{2^{2m}}^\ast$,
\begin{eqnarray*}f(\alpha)&=&\alpha^{5-2(q+1)}\left(1+\alpha^{2^m-1}+\alpha^{5\cdot(2^m-1)}\right)\\
&=&\alpha^{q^2-1+5-2(q+1)}\left(1+\alpha^{2^m-1}+\alpha^{5\cdot(2^m-1)}\right)\\
&=&\alpha^{3\cdot2^m-3}+\alpha^{2^{2m}-2^m+1}+\alpha^{2^{2m}-2^{m+1}+2}.\end{eqnarray*}
Therefore, we have the following corollary.

\begin{corollary}Let $m$ be a positive integer, $q=2^m$. Then
$$f(x) :=x^{3\cdot2^m-2}+x^{2^{2m}-2^m+1}+x^{2^{2m}-2^{m+1}+2} \in\gf_{2^{2m}}[x]$$ permutes $\gf_{2^{2m}}$ if and only if $2|m$ and $m>2$.\end{corollary}

Similarly, we have the following theorems.

\begin{theorem} Let $m>0$  and $k$ be integers, $q=2^m$. Then
$$f(x) := x^{5+k(q+1)}\left(1+x^{4\cdot(2^m-1)}+x^{5\cdot(2^m-1)}\right)\in\gf_{2^{2m}}[x]$$
permutes $\gf_{2^{2m}}^\ast$ if and only if $\gcd(5+2k,\,\, 2^m-1)=1$ and $2|m$. In particular,
$$g(x) := x^5+x^{4\cdot2^m+1}+x^{5\cdot2^m}\in\gf_{2^{2m}}[x]$$ is a permutation
trinomial over $\gf_{2^{2m}}$ if and only if $m\equiv2\pmod{4}$.\end{theorem}

\begin{theorem} Let $m>0$  and $k$ be integers, $q=2^m$. Then
$$f(x) := x^{4+k(q+1)}\left(1+x^{(2^m-1)}+x^{3\cdot(2^m-1)}\right)\in\gf_{2^{2m}}[x]$$ permutes $\gf_{2^{2m}}^\ast$ if and only if $\gcd(4+2k, 2^m-1)=1$ and $3\not|m$. In particular, $$g(x) := x^2+x^{2^{m+1}}+x^{2^{2m}-2^m+ 2}\in\gf_{2^{2m}}[x]$$ is a  permutation
trinomial over $\gf_{2^{2m}}$ if and only if $\gcd(3, \,\, m)=1$.\end{theorem}

\begin{theorem}Let $m>0$  and $k$ be integers, $q=2^m$. Then
$$f(x) := x^{2+k(q+1)}\left(1+x^{2\cdot(2^m-1)}+x^{3\cdot(2^m-1)}\right)\in\gf_{2^{2m}}[x]$$ permutes $\gf_{2^{2m}}^\ast$ if and only if $\gcd(2+2k,\,\, 2^m-1)=1$ and $3\not|m$. In particular, $$g(x) := x^{2^m+3}+x^{3\cdot2^m+1}+x^{4\cdot2^m}\in\gf_{2^{2m}}[x]$$ is a permutation
trinomial over $\gf_{2^{2m}}$ if and only if $\gcd(3, \,\, m)=1$.\end{theorem}

\section{The QM equivalence of some known permutation trinomials}

Let $m>1$ be an odd positive integer, and put $k=\frac{m+1}{2}$. Then we can prove the following congruences easily
$$(2^k-1)(2^k+1)\equiv1\pmod{2^m-1},\quad (2^k+2)(2^k-1)\equiv 2^k\pmod{2^m-1},$$
$$(2^k-2)(2^{k-1}+1)\equiv-1\pmod{2^m-1},\quad 2^k(2^{k-1}+1)\equiv 2^k+1\pmod{2^m-1}.$$

By Proposition \ref{potri1}, $f(x)=x+x^{2^k-1}+x^{2^k+1}$ is a permutation polynomial over $\gf_{2^m}$. Hence for any $b\in\gf_{2^m}^\ast$,
 $$f_1(x)=\frac{1}{b}f(bx)=x+b^{2^k-2}x^{2^k-1}+b^{2^k}x^{2^k+1}$$
is a permutation polynomial over $\gf_{2^m}$. Let $b=u^{2^m-2^{k-1}-2}$, then
$$f_1(x)=x+ux^{2^k-1}+u^{2^m-2^k-2}x^{2^k+1}$$ is a permutation polynomial over $\gf_{2^m}$ for any $u\in\gf_{2^m}^\ast$, and thus we obtain  Theorem 4.10  of \cite{LQC15}; Let $b^2=u$, then
$$f_1(x)=x+u^{2^{k-1}-1}x^{2^k-1}+u^{2^{k-1}}x^{2^k+1}$$ is a permutation polynomial over $\gf_{2^m}$ for any $u\in\gf_{2^m}^\ast$, and hence we obtain  Theorem 5.1 of \cite{MZFG15}.
% Therefore we have proved that
%\begin{corollary}  Let $m>1$ be an odd integer, and write $k=\frac{m+1}{2}$. Then for each $u\in\gf_{2^m}^\ast$,  $f(x)=x+ux^{2^k-1}+u^{2^m-2^k-2}x^{2^k+1}$ is a permutation polynomial over $\gf_{2^m}^\ast$.\end{corollary}

Similarly,
 $$f_2(x)=f(x^{2^k+2})=ux^{2^k}+x^{2^k+2}+u^{2^m-2^k-2}x^{3\cdot2^k+4},$$
$$f_3(x)= f(x^{2^k+1})=ux+x^{2^k+1}+u^{2^m-2^k-2}x^{2^{k+1}+3}$$ and
$$f_4(x)=f(x^{2^k-1})=u^{2^m-2^k-2}x+x^{2^k-1}+ux^{2^m-2^{k+1}+2}$$
are also  permutation polynomials over $\gf_{2^m}$ for any $u\in\gf_{2^m}^\ast$.

\begin{remark} Note that $g(x)=f(x^{2^k-1})/u^{2^m-2^k-2}=x+u^{2^k+1}x^{2^k-1}+u^{2^k+2}x^{2^m-2^{k+1}+2}$. Letting $u=a^{2^k-1}$, we obtain that
$$g(x)=x+ax^{2^k-1}+a^{2^k}x^{2^m-2^{k+1}+2}$$ is a permutation polynomial over $\gf_{2^m}$  for any $a\in\gf_{2^m}^\ast$, therefore we get  Theorem 5.2 of \cite{MZFG15}.  \end{remark}

From the above arguments, we have the following result.

\begin{theorem} Let $m>1$ be an odd positive integer, and write $k=\frac{m+1}{2}$. Then $f(x)=x+x^{2^k-1}+x^{2^k+1}$,\,\, $f_1(x),\,\, f_2(x),\,\, f_3(x),\,\, f_4(x),\,\, g(x),\,\, h(x)=x^{2^k}+x^{2^k+2}+x^{3\cdot2^k+4}$ are permutation trinomials over $\gf_{2^m}$. Moreover, they are QM equivalent to each other. \end{theorem}

\begin{remark} (i) Let $m>1$ be an odd positive integer, by Theorem 2.2 of \cite{DQWYY15},  $f(x)=x+x^3+x^{2^m-2^{(m+3)/2}+2}$ is a permutation polynomial over $\gf_{2^m}$, so
$$f_1(x)=\frac{f(a^{2^{m-1}}x)}{a^{2^{m-1}}}=x+ax^3+a^{2^m-2^{(m+1)/2}}x^{2^m-2^{(m+3)/2}+2}$$
is a permutation polynomial over $\gf_{2^m}$ for any $a\in\gf_q^\ast$, hence we obtain   Theorem 4.11 of \cite{LQC15}, and $f_1(x)$ is QM equivalent to $f(x)$.

(ii) Let $m>1$ be an odd positive integer, and let $f(x)=x+x^{2^{(m+1)/2}-1}+x^{2^m-2^{(m+1)/2}+1}$. Since
$$f_1(x)=f(x^{2^{(m+1)/2}+2})=x^{2^{(m+1)/2}+2}+x^{2^{(m+1)/2}}+x^2,$$
$f(x)$ is QM equivalence to $x^{2^{(m+1)/2}+2}+x^{2^{(m+1)/2}}+x^2=(x+1)^{2^{(m+1)/2}+2}+1$, which is a  permutation polynomial over $\gf_{2^m}$ since $\gcd(2^{(m+1)/2}+2,\,\, 2^m-1)=1$. Hence we have a simple proof of Theorem 2.1 of \cite{DQWYY15}.

There are also some other QM equivalent known permutation trinomials in the literature, we omit the details here.
 \end{remark}

\section*{Acknowledgements:} The authors would like to thank Michael E. Zieve for pointing out the errors in
Theorems 3.5, 4.5 and 4.6 in the original version of this paper and suggesting the corresponding corrections.
This manuscript is a corrected version of the original paper published in \textit{Fields and Their Applications 46 (2017), 38--56. }

\end{document}